\documentclass[12pt]{article}
\usepackage{amssymb, amsmath, amsfonts, amsthm, latexsym, a4}
\usepackage[latin1]{inputenc}
\usepackage[all]{xy}
\begin{document}

\newcommand{\rdg}{\hfill $\Box $}

\newtheorem{definition}{Definition}[section]
\newtheorem{theorem}[definition]{Theorem}
\newtheorem{proposition}[definition]{Proposition}
\newtheorem{corollary}[definition]{Corollary}
\newtheorem{remark}[definition]{Remark}
\newtheorem{example}[definition]{Example}
\newtheorem{lemma}[definition]{Lemma}
\newcommand{\tp}{\otimes}
\newcommand{\N}{\mathbb{N}}
\newcommand{\Z}{\mathbb{Z}}
\newcommand{\K}{\mathbb{K}}
\newcommand{\op}{\oplus}
\newcommand{\n}{\underline n}
\newcommand{\es}{{\frak S}}
\newcommand{\ef}{\frak F}
\newcommand{\qu}{\frak Q} \newcommand{\ga}{\frak g}
\newcommand{\la}{\lambda}
\newcommand{\ig}{\frak Y}
\newcommand{\te}{\frak T}
\newcommand{\cok}{{\sf Coker}}
\newcommand{\Hom}{{\sf Hom}}
\newcommand{\Image}{{\sf Im}}
\newcommand{\Ker}{{\sf Ker}}
\newcommand{\Coker}{{\sf Coker}}
\newcommand{\ext}{{\sf Ext}}
\newcommand{\id}{{\sf id}}

\newcommand{\ele}{\cal L} \newcommand{\as}{\cal A} \newcommand{\ka}{\cal K}\newcommand{\eme}{\cal M} \newcommand{\pe}{\cal P}

\newcommand{\pn}{\par \noindent}
\newcommand{\pbn}{\par \bigskip \noindent}
\bigskip\bigskip

\centerline{\bf \large A non-abelian Hom-Leibniz tensor product and applications}
\bigskip

\centerline{\bf J. M. Casas$^{(1)}$, E. Khmaladze$^{(2)}$ and N. Pacheco Rego$^{(3)}$}


\bigskip \bigskip
\centerline{$^{(1)}$Dpto.  Matem\'atica Aplicada I, Univ. de Vigo, 36005 Pontevedra, Spain}
\centerline{e-mail address: \tt jmcasas@uvigo.es}
\medskip

\centerline{$^{(2)}$A. Razmadze Math. Inst. of I. Javakhishvili Tbilisi State University,}
\centerline{Tamarashvili Str. 6, 0177 Tbilisi, Georgia}
\centerline{e-mail address: \tt e.khmal@gmail.com}
\medskip

\centerline{$^{(3)}$IPCA, Dpto. de Ciências, Campus do IPCA,
 Lugar do Aldão}
\centerline{4750-810 Vila Frescainha, S. Martinho, Barcelos,
 Portugal}
\centerline{e-mail address: \tt natarego@gmail.com}

\numberwithin{equation}{section}

\bigskip \bigskip
\begin{abstract}
The notion of non-abelian Hom-Leibniz  tensor product is introduced and some properties are established. This tensor product is used in the description of the universal ($\alpha$-)central extensions of Hom-Leibniz algebras. We also give its application to the Hochschild homology of Hom-associative algebras.
\end{abstract}

\bigskip \bigskip
{\bf A. M. S. Subject Class. (2010):} 17A30, 17B55, 17B60, 18G35, 18G60



\section{Introduction}

Since the invention of Hom-Lie algebras { as the algebraic model of deformed Lie algebras coming from twisted discretizations of vectors fields} by Hartwig, Larsson and Silvestrov in \cite{HLS}, many papers appeared dealing with  Hom-type generalizations of various algebraic structures (see  for instance \cite{AMM, BM, CIP1, ChS, ChSu, MS, MS3, MS2, Sheng, Yau, Yau1} and related references given therein). In particular, Makhlouf and Silvestrov introduced the notion of Hom-Leibniz algebras in \cite{MS}, which simultaneously is a non-commutative generalization of Hom-Lie algebras and Hom-type generalization of Leibniz algebras. In this generalized framework, it is natural to seek for possible extensions of results in the categories of Lie or Leibniz algebras to the categories of Hom-Lie or Hom-Leibniz algebras.

Recently, in \cite{CKP}, we developed the non-abelian Hom-Lie tensor product, extending the non-abelian Lie tensor product by Ellis \cite{El1} from Lie to Hom-Lie algebras.  It has applications in universal ($\alpha$-)central extensions of Hom-Lie algebras and cyclic homology of Hom-associative algebras.

In this paper we have chosen to work with Hom-Leibniz algebras. This work is a continuation and non-commutative generalization of the work already begun in  \cite{CKP}. Thus, we introduce a non-abelian Hom-Leibniz tensor product, extending the non-abelian Leibniz tensor product by Gnedbaye \cite{Gn}, which itself is the Leibniz algebra (non-commutative) version of the non-abelian Lie tensor product \cite{El1, InKhLa}. Then we investigate properties of the non-abelian Hom-Leibniz tensor product (Section \ref{section3}) and give its applications to universal ($\alpha$-)central extensions of Hom-Leibniz algebras (Section \ref{section4}) and Hochschild  homology of Hom-associative algebras (Section \ref{section5}).

One observes that not all results can be transferred from Leibniz to Hom-Leibniz algebras. For example, results on universal central extensions of Leibniz algebras can not be extended directly to Hom-Leibniz algebras  because the category of Hom-Leibniz algebras doesn't satisfy the so-called universal central extension condition \cite{CVdL}, which means that the composition of central extensions is not central in general \cite{CIP1}. By this reason, the notion of $\alpha$-central extension of Hom-Leibniz algebras is introduced in \cite{CIP1}, and classical results are divided between universal central and universal $\alpha$-central extensions of Hom-Leibniz algebras (see Theorem \ref{teorema}). Further, Hom-type version of Gnedbaye's result relating Hochschild and Milnor type Hochschild homology of associative algebras \cite{Gn}, doesn't hold for all Hom-associative algebras and requires an additional condition (see Theorem \ref{application}).

\subsection*{Notations} Throughout this paper we fix $\mathbb{K}$ as a ground field. Vector spaces are considered over $\mathbb{K}$ and linear maps are $\mathbb{K}$-linear maps. We write $\otimes$ for the tensor product $\otimes_\mathbb{K}$  over $\mathbb{K}$. For any vector space (resp. Hom-Leibniz algebra) $L$, a subspace (resp. a two-sided ideal) $L'$ and $x\in L$ we write $\overline{x}$ to denote the coset $x+L'$. We denote by {\sf{Lie}} and {\sf{Lb}} the categories of Lie and Leibniz algebras, respectively.


\section{Preliminaries on Hom-Leibniz algebras}

\subsection{Basic definitions}

In this section we review some terminology on Hom-Leibniz algebras and recall notions used in the paper. We also
 introduce notions of actions and semi-direct product of  Hom-Leibniz algebras.

\begin{definition}\label{HomLeib} \cite{MS}
A Hom-Leibniz algebra is a triple $(L,[-,-],\alpha_L)$ consisting of a vector space $L$, a bilinear map $[-,-] : L \times L \to L$, called bracket operation, and a linear map $\alpha_L : L \to L$ satisfying:
\begin{equation} \label{def}
 [\alpha_L(x),[y,z]]=[[x,y],\alpha_L(z)]-[[x,z],\alpha_L(y)] \ \ ({\text{Hom-Leibniz identity}})
\end{equation}
for all $x, y, z \in L$.
\end{definition}

In the whole paper we only deal with (the so-called multiplicative \cite{Yau1}) Hom-Leibniz algebras  $(L,[-,-],\alpha_L)$ such that $\alpha_L$ preserves the bracket operation, that is, $\alpha_L [x,y] = [\alpha_L(x),\alpha_L(y)]$, for all $x, y \in L$.

\begin{example}\label{ejemplo 1} \
\begin{enumerate}
\item[a)] Taking $\alpha = \id$ in Definition \ref{HomLeib}, we obtain the definition of a Leibniz algebra \cite{Lo}. Hence any Leibniz algebra $L$ can be considered as a Hom-Leibniz algebra with $\alpha_L=\id$.

\item[b)] Any Leibniz algebra $L$ can be considered as a Hom-Leibniz algebra with $\alpha_L=0$.  In fact, any vector space L endowed with any bracket operation is a Hom-Leibniz algebra with $\alpha_L=0$.
    \item[c)] Any  Hom-Lie algebra \cite{HLS} is a Hom-Leibniz algebra whose bracket operation satisfies the skew-symmetry condition.

    \item[d)] Any Hom-dialgebra \cite{Yau} $(D,\dashv,\vdash, \alpha_D)$ becomes a Hom-Leibniz algebra
    $(D,$ $[-,-],\alpha_D)$ with the bracket given by $[x,y]=x \dashv y - y \vdash x$, for all $x,y \in D$.

  \item[e)] Let $(L,[-,-])$ be a Leibniz algebra and $\alpha_L:L \to L$  a Leibniz algebra endomorphism. Define $[-,-]_{\alpha} : L \otimes L \to L$ by $[x,y]_{\alpha} = [\alpha(x),\alpha(y)]$, for all $x, y \in L$. Then $(L,[-,-]_{\alpha}, \alpha_L)$ is a Hom-Leibniz algebra.

\item[f)] Any Hom-vector space $(V,\alpha_V)$, (i.e. $V$ is a vector space and $\alpha_V:V\to V$ is a linear map) together with the trivial bracket $[-,-]$ (i.e. $[x,y]=0$ for all $x,y\in V$) is a Hom-Leibniz algebra, called {abelian  Hom-Leibniz algebra}.

\item[g)]  { The two-dimensional $\mathbb{C}$-vector space $L$ with basis $\{ a_1, a_2 \}$, endowed with the bracket operation  $[a_2, a_2]= a_1$ and zero elsewhere, and the endomorphism $\alpha_L$ given by the matrix $\left( \begin{array}{cc} 1 & 1 \\0 & 1 \end{array} \right)$ is a non-Hom-Lie Hom-Leibniz algebra.}
\end{enumerate}
\end{example}

In the sequel we shall use the shortened notation $(L,\alpha_L)$ for $(L,[-,-],\alpha_L)$.
\begin{definition}\label{homo}
A homomorphism of Hom-Leibniz algebras $f:(L,\alpha_L) \to (L',\alpha_{L'})$ is a linear map $f : L \to L'$ such that
\begin{align*}
&f([x,y]) =[f(x),f(y)],\\
&f \circ \alpha_L = \alpha_{L'} \circ  f,
\end{align*}
for all $x, y \in L$.
\end{definition}

Hom-Leibniz algebras and  their  homomorphisms form  a category, which we denote by ${\sf{HomLb}}$. Example \ref{ejemplo 1} a), b) and c) say that there are full embedding functors \[
I_0,I_1:{\sf{Lb}}\longrightarrow {\sf{HomLb}}, \quad I_0(L)= (L,0), \ I_1(L)= (L,\id)
 \]
 and
 \[
 {\sf{inc}:\sf{HomLie}}\longrightarrow {\sf{HomLb}}, \quad (L,\alpha_L)\mapsto (L,\alpha_L),
 \]
 where ${\sf{HomLie}}$ denotes the category of Hom-Lie algebras. Note that the restrictions of $I_0$ and $I_1$ to the category of Lie algebras are full embeddings
 \[
I'_0,I'_1:{\sf{Lie}}\longrightarrow {\sf{HomLie}}, \quad I'_0(L)= (L,0), \ I'_1(L)= (L,\id).
 \]

\begin{definition}
 A  Hom-Leibniz subalgebra $(H, \alpha_H)$ of a Hom-Leibniz algebra $(L,\alpha_L)$ consists of a vector subspace $H$ of $L$, which is closed under the bracket, together with the linear map $\alpha_H:H\to H$ being the restriction of $\alpha_L$ on $H$. In such a case we may write $\alpha_{L \mid}$ for $\alpha_H$.

A  Hom-Leibniz subalgebra $(H, \alpha_H)$ of $(L, \alpha_L)$ is said to be a  two-sided Hom-ideal if $[x,y], [y,x] \in H$, for all $x \in H, y \in L$.

If $(H,\alpha_H)$ is a two-sided  Hom-ideal of $(L,\alpha_L)$, then the quotient vector space $L/H$ together with the endomorphism $\overline{\alpha}_L : L/H  \to L/H$ induced by $\alpha_L$, naturally inherits a structure of Hom-Leibniz algebra, and it is called the quotient Hom-Leibniz algebra.

The commutator of two-sided Hom-ideals
$(H, \alpha_{L \mid})$ and $(K, \alpha_{L\mid })$ of a Hom-Leibniz algebra $(L,\alpha_L)$, denoted by $([H,K],\alpha_{L\mid })$, is the Hom-Leibniz subalgebra of $(L,\alpha_L)$ spanned by the brackets $[h,k]$ and $[k, h]$ for all $h \in H$, $k \in K$.
\end{definition}

The following lemma can be readily checked.
\begin{lemma}\label{ideales} Let $(H,\alpha_{H})$ and $(K,\alpha_{K})$ be two-sided Hom-ideals of a Hom-Leibniz algebra
$(L,\alpha_L)$. The following statements hold:

\begin{enumerate} \label{ideales}

\item[a)] $(H \cap K, \alpha_{L \mid})$ and $(H + K, \alpha_{{L \mid}})$ are two-sided Hom-ideals of $(L,\alpha_{L})$;

\item[b)]  $[H,K] \subseteq H \cap K$;

\item[c)] If $\alpha_{L}$ is surjective, then $([H,K],\alpha_{L \mid})$ is a two-sided Hom-ideal of $(L,\alpha_L)$;

\item[d)] $([H,K],\alpha_{L \mid})$ is a two-sided Hom-ideal of $(H,\alpha_{H})$  and $(K,\alpha_{K})$. In particular, $([L,L],\alpha_{L \mid})$ is a two-sided Hom-ideal of $(L,\alpha_L)$;

\item[e)] $(\alpha_L(L),\alpha_{L {\mid}})$ is a Hom-Leibniz subalgebra of $(L,\alpha_{L})$;

\item[f)] If $H, K \subseteq \alpha_{L}(L)$, then $([H, K],\alpha_{L \mid})$ is a two-sided Hom-ideal of $(\alpha_{L}(L),\alpha_{{L \mid}})$.
\end{enumerate}
\end{lemma}

\begin{definition}
Let $(L,\alpha_L)$ be a Hom-Leibniz algebra. The subspace $Z(L) = \{ x \in L \mid [x, y] =0 = [y,x], \text{for\ all}\ y \in L \}$ of $L$ is said to be the center of $(L,\alpha_L)$.
\end{definition}

Note that if $\alpha_L : L \to L$ is a surjective homomorphism, then  $(Z(L),\alpha_{L \mid})$ is a two-sided Hom-ideal of $L$.

\begin{corollary}\label{corollary}
Any Hom-Leibniz algebra  $(L,\alpha_L)$ gives rise to a Hom-Lie algebra $(L_{\sf{Lie}}, \overline\alpha_L)$, which is obtained as the quotient of $L$ by the relation $[x,x]=0$, $x\in L$. Here $\overline\alpha_L$ is induced by $\alpha_L$.
This defines a functor
$
 (-)_{\sf{Lie}}:\sf{HomLb}\longrightarrow {\sf{HomLie}}.
$
Moreover, the canonical epimorphism   $(L,\alpha_L) \twoheadrightarrow (L_{\sf{Lie}}, \overline\alpha_L)$ is universal among all homomorphisms from { $(L,\alpha_L)$} to a Hom-Lie algebra, implying that the functor $(-)_{\sf{Lie}}$ is left adjoint to the inclusion functor $ {\sf{inc}:\sf{HomLie}}\longrightarrow {\sf{HomLb}} $.
This adjunction is a natural extension of the well-known adjunction
$
\xymatrix  {
\sf{Lie} \ar@<-1.1ex>[r]^{\perp}_{\sf{inc}}  & \sf{Lb}  \ar@<-1.1ex>[l]_{(-)_{\sf{Lie}}}
}
$
between the categories of Lie and Leibniz algebras,  in the sense that the following inner and outer diagrams
\[
\xymatrix @=20mm {
	\sf{Lie} \ar@<-1.1ex>[r]^{\perp}_{\sf{inc}} \ar[d]_{I'_i} & \sf{Lb}\ar@<-1.1ex>[l]_{(-)_{\sf{Lie}}} \ar[d]^{I_i} \\
	\sf{HomLie} \ar@<1.1ex>[r]_{\top}^{\sf{inc}}  & \sf{HomLb}  \ar@<1.1ex>[l]^{(-)_{\sf{Lie}}}
}
\]
are commutative for $i=0,1$.
\end{corollary}

\subsection{Hom-Leibniz actions and semi-direct product}

\begin{definition} \label{Hom accion} Let $\left(  L ,\alpha_{L}\right)$ and $\left(  M,\alpha_{M}\right)$ be Hom-Leibniz algebras. A Hom-Leibniz action of $\left(  L, \alpha_{L}\right)$ on $\left(  M, \alpha_{M}\right)$ is a couple of linear maps,
$L\otimes M\to M$, $ x\otimes m \mapsto {^x}  m$ and $M\otimes L\to M$, $ m\otimes x  \mapsto m^  x$, satisfying the following identities:
\begin{enumerate}
\item[a)] $\alpha_{M}\left(  m\right)  ^{  \left[  x,y\right]}  =\left(
m ^  x\right)  ^ {\alpha_{L}\left(  y\right)}  -\left(  m ^  y\right)  ^{ \alpha_{L}\left(  x\right)},$

\item[b)] $  {^{\left[  x,y\right]}}  \alpha
_{M}\left(  m\right)= \left({^x}  m\right)  ^  {\alpha_{L}\left(  y\right)}  -  {^{\alpha_{L}\left(  x\right)}}   \left(  m ^ y\right),$

\item[c)] ${^{\alpha_{L}\left(  x\right)}}  \left(  {^y} m\right)  =   -{^{\alpha_{L}\left(  x\right)}}   \left(  m ^ y\right),$

\item[d)] ${^{\alpha_{L}\left(  x\right)}}  \left[  m,m^{\prime}\right]  = \left[
 {^x}m,\alpha_{M}\left(  m^{\prime}\right)  \right]  -\left[
 {^x}m^{\prime},\alpha_{M}\left(  m\right)  \right],$

\item[e)] $ \left[  m,{m^{\prime}}\right]  {^{\alpha_{L}\left(  x\right)}}   =  \left [ m ^ x,\alpha_{M}\left(  m^{\prime}\right)  \right] + \left[  \alpha_{M}\left(  m\right)  ,  {m^{\prime}} ^{\ x} \right ],$

\item[f)] $\left[  \alpha_{M}\left(  m\right)  ,  {^x}  {m^{\prime}}
\right]  = - \left[  \alpha_{M}\left(  m\right)  ,  {m^{\prime}} ^{\ x} \right ],$

\item[g)] $\alpha_{M}\left(  {^x}m\right)  = {^{\alpha_{L}\left(  x\right)}} \alpha_{M}\left(  m\right),$

\item[h)] $\alpha_{M}\left(  m^x\right)  =\alpha_{M}\left(  m\right) ^{ \alpha_{L}\left(  x\right)},$
\end{enumerate}
for all $x,y\in L$ and $m,m^{\prime}\in M$.

The  action is called trivial if ${^x}m=0=m^x$, for all $x\in L$, $m\in M$.
\end{definition}

\begin{example}\label{Hom accion Leib} \
\begin{enumerate}
\item[a)] A Hom-representation of a Hom-Leibniz algebra $(L,\alpha_L)$ is  a Hom-vector space $(M, \alpha_M)$ equipped with two linear operations, $L\otimes M\to M$, $ x\otimes m \mapsto {^x}  m$ and $M\otimes L\to M$, $ m\otimes x  \mapsto m^  x$, satisfying the axioms a), b), c), g), h) of Definition \ref{Hom accion}. Therefore, representations over a Hom-Leibniz algebra $(L,\alpha_L)$ are abelian Leibniz algebras enriched with Hom-Leibniz actions of $(L,\alpha_L)$.

\item[b)] Any action of a Leibniz algebra $L$ on another Leibniz algebra $M$ (see e.g.  \cite{LP}) gives a Hom-Leibniz action of $(L,\id_L)$ on $(M,\id_M)$.

\item[c)]  Let $\left( K, \alpha_{K}\right)$ be a Hom-Leibniz subalgebra of a  Hom-Leibniz algebra $\left(  L, \alpha_{L}\right)$
and $\left(  H, \alpha_{H}\right)$ be a two-sided Hom-ideal of $\left(  L, \alpha_{L}\right)$. There exists a Hom-Leibniz action of $\left(  K,\alpha_{K}\right)$ on $\left(  H, \alpha_{H}\right)$ given by the bracket in $\left(  L,\alpha_{L}\right).$

 \item[d)] Let $0 \to (M,\alpha_M) \stackrel{i}\to (K,\alpha_K) \stackrel{\pi}\to (L,\alpha_L) \to 0$ be a split short exact sequence of Hom-Leibniz  algebras, i. e.  there exists a homomorphism of Hom-Leibniz algebras $s:(L,\alpha_L)\to (K,\alpha_K)$ such that $\pi \circ s=\id_L$. Then there is a Hom-Leibniz action of $(L,\alpha_L)$ on $(M,\alpha_M)$ defined in the standard way:  ${}^{x}m=i^{-1}[s(x),i(m)]$ and $m^x=i^{-1}[i(m),s(x)]$ for all $x\in L$, $m\in M$.

\end{enumerate}
\end{example}

\begin{definition} \label{producto semidirecto}
Let $\left(  M,\alpha_{M}\right)$ and  $\left(  L,\alpha_{L}\right)$ be Hom-Leibniz algebras together with a Hom-Leibniz action of $\left(  L,\alpha_{L}\right)$ on $\left(  M,\alpha_{M}\right)$.   Their semi-direct product $\left(  M\rtimes L,{\alpha_{\rtimes}}\right)$ is the Hom-Leibniz algebra with  the underlying vector space $M\oplus L$,  endomorphism ${\alpha_{\rtimes}}:M\rtimes L\to M\rtimes L$ given by ${\alpha_{\rtimes}} \left(  m,l\right) = \left(  \alpha_{M}\left(m\right)  ,\alpha_{L}\left(  l\right)  \right)$ and bracket
$$\left[  \left(  m_{1},l_{1}\right)  ,\left(  m_{2},l_{2}\right)
\right]  =\left( [m_1,m_2] +  {^{\alpha_{L}\left(  l_{1}\right)}}  m_{2}+m_{1}^{
\alpha_{L}\left(  l_{2}\right)}  ,\left[  l_{1},l_{2}\right]  \right).$$
\end{definition}

Let $\left(  M,\alpha_{M}\right)$ and  $\left(  L,\alpha_{L}\right)$ be Hom-Leibniz algebras with a Hom-Leibniz action of $\left(  L,\alpha_{L}\right)$ on $\left(  M,\alpha_{M}\right)$. Then we have the following short exact sequence of Hom-Leibniz algebras
\begin{equation} \label{extension semidirectoLeib}
0\to \left(  M,\alpha_{M}\right)
\overset{i}{\to}\left(  M\rtimes L,{\alpha_{\rtimes}}\right)
\overset{\pi}{\to}\left(  L,\alpha_{L}\right)  \to 0\ ,
\end{equation}
where $i: M \to  M\rtimes L$, $i(m) =  \left(  m,0\right)$, and $\pi:  M\rtimes L  \to  L$,
$\pi \left(  m,x\right)   = x$. Moreover, this sequence splits by the Hom-Leibniz homomorphism $L  \to M\rtimes L$,
$ x\mapsto  \left(  0,x\right)$.

\subsection{Homology of Hom-Leibniz algebras}
In this subsection we recall from \cite{CIP1} the construction of homology vector spaces of a Hom-Leibniz algebra with coefficients in a Hom-co-representation. 

\begin{definition}
 A Hom-co-representation of a Hom-Leibniz algebra $(L,\alpha_L)$ is a Hom-vector space $(M,\alpha_M)$ together with two linear maps $ L \otimes M \to M$, $x\otimes m \mapsto {^x} m$ and $M \otimes L \to M$, $m \otimes x = m ^ x$, satisfying the following identities:
\begin{enumerate}
\item[a)] ${^{[x,y]}} \alpha_M(m) = {^{\alpha_L(x)}}({^y} m) - {^{\alpha_L(y)}} ({^x} m)$,

\item[b)] $\alpha_M(m) ^{ [x,y]}  = ({^y} m) ^{ \alpha_L(x)} - {^{\alpha_L(y)}} (m ^x)$,

\item[c)] $(m ^ x) ^{ \alpha_L(y)} = - {^{\alpha_L(y)}} (m ^x)$,

\item[d)] $\alpha_M({^x} m) = {^{\alpha_L(x)}} \alpha_M(m)$,

\item[e)] $\alpha_M(m ^ x) =  \alpha_M(m) ^{\alpha_L(x)}$,
\end{enumerate}
for any $x, y \in L$ and $m \in M$
\end{definition}

\begin{example}\ \label{ejemplo}
\begin{enumerate}
\item[a)] { Let} $M$ be a  co-representation of a Leibniz algebra L \cite{LP}, then $(M,\id_M)$ is a  Hom-co-representation of the  Hom-Leibniz algebra $(L,\id_L)$.

\item[b)] The underlying Hom-vector  space of a  Hom-Leibniz algebra  $(L,\alpha_L)$ has a Hom-co-representation structure  given by
 ${^x} y =-[y, x]$ and  $y ^ x = [y, x]$, $x,y\in L$.
\end{enumerate}
\end{example}

 Let $(L,\alpha_L)$ be a Hom-Leibniz algebra and $(M,\alpha_M)$ be a Hom-co-representation of $(L,\alpha_L)$. The homology $HL_{\star}^{\alpha}(L,M)$ of $(L,\alpha_L)$ with coefficients in $(M,\alpha_M)$ is defined to be the homology of the chain complex $(CL_{\star}^{\alpha}(L,M),d_{\star})$, where
  \[
  CL_n^{\alpha}(L,M) :=M \otimes L^{\otimes n}, \quad n\geq 0
  \]
  and $d_n: CL_n^{\alpha}(L,M) \to CL_{n-1}^{\alpha}(L,M)$, $n\geq 1 $, is the linear map given by
    \begin{align*}
    d_n&(m \otimes x_1 \otimes \dots \otimes x_n)= m^{x_1} \otimes \alpha_L(x_2) \otimes \dots \otimes \alpha_L(x_n) \\
    & +   \sum_{i=2}^n (-1)^i \ \ {^{x_i}}m \otimes \alpha_L(x_1) \otimes \dots \otimes \widehat{\alpha_L(x_i)} \otimes \dots \otimes \alpha_L(x_n) \\
    &+ \sum _{1 \leq i < j \leq n} (-1)^{j+1} \alpha_M(m) \otimes \alpha_L(x_1) \otimes \dots \otimes \alpha_L(x_{i-1}) \otimes [x_i,x_j] \otimes \alpha_L(x_{i+1}) \\
    & \qquad \qquad \qquad  \qquad \qquad \qquad \qquad \qquad \qquad \otimes \dots \otimes  \widehat{\alpha_L(x_j)} \otimes \dots \otimes \alpha_L(x_n),    \end{align*}
where the notation $\widehat{\alpha_L(x_i)}$ indicates that the variable $\alpha_L(x_i)$ is omitted.
Hence
\[
HL_{n}^{\alpha}(L,M) :=H_{n}(CL_{\star}^{\alpha}(L,M), d_{\star}), \quad n\geq 0.
\]

Direct calculations show that $HL^{\alpha}_0(L,M) = {M}/{M^L}$, where $M^L=\{m^x \mid m \in M, x \in L\}$, and if $(M,\alpha_M )$ is a trivial Hom-co-representation of $(L,\alpha_L )$, that is $m ^ x = {^x} m = 0$, then $HL^{\alpha}_1(L,M) = \left({M \otimes L}\right) /\left({\alpha_M(M) \otimes [L,L]}\right)$. In particular, if $M=\K$ then $HL^{\alpha}_1(L,\K)=L/[L,L]$. Later on we write  $HL^{\alpha}_n(L)$ for $HL^{\alpha}_n(L,\K)$ { and it is called homology with trivial coefficients}.


\section{Non-abelian  Hom-Leibniz tensor product}\label{section3}

Let $(M,\alpha_M)$ and $(N,\alpha_N)$ be Hom-Leibniz algebras acting on each other. We denote by $M\ast N$  the vector space spanned by all symbols $m \ast n$, $n \ast m$ and subject to the following relations:
\begin{align}\label{a)}
\begin{array}{crcl}
&\lambda (m \ast n) &= &(\lambda m) \ast n= m \ast ( \lambda n), \\
&\lambda (n \ast m) &=&b (\lambda n) \ast m= n \ast ( \lambda m), \\
&(m+m') \ast n &= &m \ast n + m' \ast n, \\
&m \ast (n + n') &= &m \ast n + m \ast n', \\
&(n+n') \ast m &= &n \ast m + n' \ast m, \\
&n \ast (m + m') &= &n \ast m + n \ast m', \\
&\alpha_{M}(  m)  \ast\left[  n,n^{\prime}\right]&= &m ^n  \ast\alpha_{N}(  n^{\prime})  -m^{n^{\prime}}  \ast\alpha_{N}(  n), \\
&\alpha_{N}(  n)  \ast\left[  m,m^{\prime}\right]  &=&n^m  \ast\alpha_{M}(  m^{\prime})  - n^{m^{\prime}}  \ast\alpha_{M}(  m),  \\
&\left[  m,m^{\prime}\right] \ast \alpha_{N}( n)&=& {}^mn \ast\alpha_{M}(  m^{\prime})  - \alpha_{M}(  m) \ast n^{m^{\prime}},  \\
&\left[  n,n^{\prime}\right] \ast \alpha_{M}( m)&=& {}^nm \ast\alpha_{N}(  n^{\prime})  - \alpha_{N}(  n) \ast m^{n^{\prime}}, \\
&\alpha_{M}(  m)  \ast  {^{m^{\prime}}n}  & = &- \alpha_{M} ( m )  \ast  n ^{m'} , \\
&\alpha_{N}(  n)  \ast  {^{n^{\prime}}m} & =& - \alpha_{N}(  n )  \ast  m ^{n'} ,  \\
&m^n\ast{{}^{m^{\prime}}}{n^{\prime}}& =& {}^mn\ast{m^{\prime}}^{n^{\prime}}, \\
&m^n \ast {n^{\prime}}^{m^{\prime}} & = & {}^mn \ast {}^{n^{\prime}}{m^{\prime}},  \\
&^n m \ast {}^{m^{\prime}}{n^{\prime}} & =& n^m \ast {m^{\prime}}^{n^{\prime}},  \\
&^n m  \ast {n^{\prime}}^{m^{\prime}} &= &n^m  \ast {}^{n^{\prime}} {m^{\prime}},
\end{array}
\end{align}
for all $\lambda \in \K$, $m, m' \in M$, $n, n' \in N$.

We claim that $(M \ast N, \alpha_{M\ast N})$ is a Hom-vector space, where $\alpha_{M\ast N}$ is the linear map induced by $\alpha_M$ and $\alpha_N$, i.e.
\begin{align*}
&\alpha_{M\ast N}\left(  m\ast n\right)  =\alpha_{M}\left(m\right)  \ast\alpha_{N}\left(  n\right), \quad
\alpha_{M\ast N}\left(  n\ast m\right)  =\alpha_{N}\left(n\right)  \ast\alpha_{M}\left(  m\right).
\end{align*}
Indeed, it can be checked readily that  $\alpha_{\ast}$ preserves all the relations in (\ref{a)}).

To be able to introduce the non-abelian Hom-Leibniz tensor product, we need to assume that the actions are compatibly in the following sense.

\begin{definition}\label{compatibility1}
Let $(M,\alpha_M)$ and $(N,\alpha_N)$ be Hom-Leibniz algebras with Hom-Leibniz actions on each other. The actions are  said to be compatible if
\begin{align}\label{compatibility}
\begin{array}{crclcrcl}
&^{(^mn)} m'&=&[m^n,m'],  && ^{(^nm)}n'&=&b[n^m,n'],\\
&^{(n^m)} m'&=&[^nm,m'],  && ^{(m^n)}n'&=& [^mn,n'],\\
& m^{(^{m^{\prime}}n)}&=&[m,m'{\ ^n}],  && n^{(^{n^{\prime}}m)}&=& [n,n'{\ ^m}],\\
& m^{(n^{m^{\prime}})}&=&[m,^nm'],  && n^{(m^{n^{\prime}})}&=& [n,^mn'],
\end{array}
\end{align}
for all $m,m'\in M$ and $n,n'\in N$.
\end{definition}

\begin{example}\label{compatible}
If  $(H,\alpha_H)$ and $(H',\alpha_{H'})$ both are Hom-ideals of a Hom-Leibniz algebra $(L,\alpha_L)$, then the Hom-Leibniz actions of $(H,\alpha_H)$ and $(H',\alpha_{H'})$ on each other, considered in Example \ref{Hom accion Leib} {c)}, are compatible.
\end{example}

Now we have the following  property:

\begin{proposition}
Let $(M,\alpha_M)$ and $(N,\alpha_N)$ be Hom-Leibniz algebras acting compatibly on each other, then the Hom-vector space
 $  (M\ast N, {\alpha}_{M\ast N} )$ endowed  with the following bracket operation
\begin{align} \label{bracket}
\begin{array} {crcl}
& \left[  m\ast n,m^{\prime}\ast n^{\prime}\right]  &=& m^n\ast^{m^{\prime}}{n^{\prime}},\\
&\left[  m\ast n,n^{\prime}\ast m^{\prime}\right]  &=& m^n \ast {n^{\prime}}^{m^{\prime}},\\
&\left[  n\ast m,m^{\prime}\ast n^{\prime}\right]  &=& {^n} m \ast ^{m^{\prime}}{n^{\prime}},\\
&\left[  n\ast m,n^{\prime}\ast m^{\prime}\right]  &=& {^n} m  \ast {n^{\prime}}^{m^{\prime}},
\end{array}
\end{align}
is a Hom-Leibniz algebra.
\end{proposition}
\begin{proof}
Routine calculations show that the bracket given by (\ref{bracket}) is compatible with the defining relations in (\ref{a)}) of $M\ast  N$ and can be extended from generators to any elements. The verification of the Hom-Leibniz identity (\ref{def}) is straightforward by using compatibility conditions in (\ref{compatibility}). Finally, it follows directly by definition of $\alpha_{M\ast N}$ that it  preserves the bracket given by (\ref{bracket}).
\end{proof}

\begin{definition}
The above Hom-Leibniz  algebra structure on $(  M\ast  N,\alpha_{M \ast N})$ is called the non-abelian Hom-Leibniz tensor product of the Hom-Leibniz algebras $\left(M,\alpha_{M}\right)$ and $\left(  N,\alpha_{N}\right)$.
\end{definition}

\begin{remark}
If $\alpha_M=\id_M$ and $\alpha_N=\id_N$, then $M \ast N$ coincides with the non-abelian tensor product of Leibniz algebras introduced in \cite{Gn}.
\end{remark}

\begin{remark}\label{remarkLieLb} Let $(M,\alpha_M)$ and $(N,\alpha_N)$ be Hom-Lie algebras. One can readily check that the following assertions hold:
\begin{enumerate}
\item[a)] Any Hom-Lie action of $(M,\alpha_M)$ on $(N,\alpha_N)$, $M\otimes N\to N$, $m\otimes n\mapsto {}^mn$ (see \cite[Definition 1.7]{CKP} for the definition), gives a Hom-Leibniz action   of $(M,\alpha_M)$ on $(N,\alpha_N)$ by letting $n^{m}= -{}^mn$ for all $m\in M$ and $n\in N$.
\item[b)] For compatible Hom-Lie actions (see \cite[Definition 2.1]{CKP}) of $(M,\alpha_M)$ and $(N,\alpha_N)$ on each other, the induced Hom-Leibniz actions are also compatible.
\item[c)] If $(M,\alpha_M)$ and $(N,\alpha_N)$ act compatibly on each other, then there is an epimorphism of Hom-Leibniz algebras
$(M\ast N, \alpha_{M\ast N}) \twoheadrightarrow (M\star N, \alpha_{M\star N})$
defined on generators by $m\ast n\mapsto m\star n$ and $n\ast m\mapsto - m\star n$, where $\star$ denotes the non-abelian Hom-Lie tensor product (see \cite[Definition 2.4]{CKP}).
\end{enumerate}
\end{remark}

Sometimes the non-abelian Hom-Leibniz tensor product can be described as the tensor product of vector spaces. In particular, we have the following:

\begin{proposition} \label{tensor abel}
If the Hom-Leibniz algebras $(M,\alpha_M)$ and $(N, \alpha_N)$ act trivially on each other and both $\alpha_M$, $\alpha_N$ are epimorphisms, then there is an isomorphism of abelian Hom-Leibniz algebras
 \[
( M\ast N, \alpha_{M\ast N}) \cong
 \big((M^{ab}\otimes N^{ab}) \oplus (N^{ab}\otimes M^{ab}),\alpha _{\oplus}\big),
\]
where $M^{ab}=M/[M,M]$, $N^{ab}=N/[N,N]$ and $\alpha_{\oplus}$ denotes the linear self-map of $(M^{ab}\otimes N^{ab}) \oplus ( N^{ab}\otimes M^{ab})$ induced by $\alpha_M$ and $\alpha_N$.
\end{proposition}
\begin{proof}
 Since the actions are trivial, then relations (\ref{bracket})  enables us to see that $(M\ast N,$ $\alpha_{M\ast N})$ is an abelian Hom-Leibniz algebra.

Since $\alpha_{M}$ and $\alpha_{N}$ are epimorphisms, the defining relations (\ref{a)}) of the non-abelian tensor product say precisely that the vector space $M\ast N$ is the quotient of $(M\otimes N)\oplus$ $(N\otimes M)$ by the relations
\[
m\ast [ n,n' ]=[m,m']\ast n =[n,n']\ast m=n\ast [ m,m' ]=0
\]
 for all $m,m' \in M$, $n,n' \in N$. The later is isomorphic to $(M^{ab}\otimes N^{ab}) \oplus (N^{ab}\otimes M^{ab})$ and this
isomorphism commutes with the endomorphisms $\alpha_{\oplus}$ and $\alpha_{M\ast N}$.
\end{proof}

The non-abelian Hom-Leibniz tensor product is functorial in the following sense: if $f:(M,\alpha_M)\to (M',\alpha_{M'})$ and $g:(N,\alpha_N)\to (N',\alpha_{N'})$ are homomorphisms of Hom-Leibniz algebras together with compatible actions of $(M, \alpha_M)$ (resp. $(M', \alpha_{M'}))$  and $(N, \alpha_N)$ (resp. $(N', \alpha_{N'}))$ on each other such that $f$, $g$ preserve these actions, that is

\[
\begin{array}{rclcrcl}
f({}^nm)&=&{}^{g(n)}f(m), & & f(m^n)&=&f(m)^{g(n)}, \\
 g({}^mn)&=&{}^{f(m)}g(n), &  & g(n^m)&=&g(n)^{f(m)},
 \end{array}
\]
for all  $m\in M,\ n\in N$, then we have a homomorphism of Hom-Leibniz algebras
\[
f\ast g:(M\ast N,\alpha_{M\ast N})\to (M'\ast N',\alpha_{M'\ast N'})
\]
 defined by $(f\ast g)(m\ast n)=f(m)\ast g(n), (f\ast g)(n\ast m)=g(n)\ast f(m)$.

\begin{proposition}\label{exact-tensor-1}
Let $0\to (M_1,\alpha_{M_1})\overset{f}{\to}(M_2,\alpha_{M_2})\overset{g}{\to}(M_3,\alpha_{M_3})\to 0$ be a short exact sequence of Hom-Leibniz algebras. Let $(N,\alpha_{N})$ be a Hom-Leibniz algebra together with compatible Hom-Leibniz actions of $(N,\alpha_{N})$ and $(M_i,\alpha_{M_i})$ $(i=1,2,3)$ on each other and $f$, $g$ preserve these actions.
Then there is an exact sequence of Hom-Leibniz algebras
\[
(M_1\ast N,\alpha_{M_1\ast N})\stackrel{f\ast \id_N}{\longrightarrow}(M_2\ast N,\alpha_{M_2\ast N})\stackrel{g\ast \id_N}{\longrightarrow}(M_3\ast N,\alpha_{M_3\ast N})\longrightarrow 0.
\]
\end{proposition}
\begin{proof}
 Clearly $g\ast \id_{N}$ is an epimorphism and ${\Image}\left( f\ast
\id_{N}\right) \subseteq \Ker\left( g\ast \id_{N}\right)$.
Since ${\Image}\left( f\ast \id_{N}\right)$ is generated by all elements of the form $f(m_{1})\ast n$, $n \ast f(m_1)$, with $m_{1}\in M_{1}$, $n\in N$, it is a two-sided Hom-ideal in $(M_{2}\ast N,\alpha_{M_{2}\ast N})$ because of the following equalities:
\[
\begin{array}{lccccl}
\left[ f(m_{1})\ast n,m_{2}\ast n'\right] &=&f(m_{1})^{n}\ast
{^{m_{2}}}n'&=&f\left( m_{1}^{\text{ \ }n}\right) \ast {^{m_{2}}}n' &\in
{\Image}\left( f\ast \id_{N}\right),\\

\left[ f(m_{1})\ast n,n'\ast m_{2}\right] &=& f(m_{1})^{n}\ast
{n'}^{\text{  \ }m_{2}} &=& f\left( m_{1}^{\text{ \ }n}\right) \ast
{n'}^{\text{ \ }m_{2}} &\in {\Image}\left( f\ast \id_{N}\right),\\

\left[ n\ast f(m_{1}),m_{2}\ast n'\right] &=& {^{n}}f(m_{1})\ast
{^{m_{2}}}n' &=& f\left( {^{n}}m_{1}\right) \ast {^{m_{2}}}n' & \in {\Image}
\left( f\ast \id_{N}\right),\\

\left[ n\ast f(m_{1}),n'\ast m_{2}\right] &=& {^{n}}f(m_{1})\ast
n'^{\text{ \ }m_{2}} &=& f\left( {^{n}}m_{1}\right) \ast {n'}^{\text{
\ }m_{2}} & \in {\Image}\left( f\ast \id_{N}\right),\\

\left[ m_{2}\ast n',f(m_{1})\ast n\right] &=& m_{2\text{ }}^{\text{ \ }n'}\ast ^{f(m_{1})}n &=& m_{2\text{ }}^{\text{ \ }n'}\ast f\left(
^{m_{1}}n\right) & \in {\Image}\left( f\ast \id_{N}\right),\\

\left[ m_{2}\ast n',n\ast f(m_{1})\right] &=& m_{2\text{ }}^{\text{ \ }n'}\ast n^{f(m_{1})} &=& m_{2\text{ }}^{n'}\ast f\left(
n^{\text{ }^{m_{1}}}\right) & \in {\Image}\left( f\ast \id_{N}\right),\\

\left[ n'\ast m_{2},f(m_{1})\ast n\right] &=& {^{n'}}m_{2}\ast
^{f(m_{1})}n &=& {^{n'}}m_{2}\ast f\left( ^{m_{1}}n\right) & \in {\Image}
\left( f\ast \id_{N}\right),\\

\left[ n'\ast m_{2},n\ast f(m_{1})\right] &=& {^{n'}}m_{2}\ast
n^{f(m_{1})} &=& {^{n'}}m_{2}\ast f\left( n^{\text{ }^{m_{1}}}\right)
& \in {\Image}\left( f\ast \id_{N}\right),
\end{array} \]
 \[
 \begin{array}{lccl}
\alpha _{M_{2}\ast N}\left( f(m_1)\ast n\right) &=& f(\alpha _{M_{2}}\left(
m_1\right) )\ast \alpha _{N}\left( n\right) & \in {\Image}\left( f\ast \id_{N}\right),  \\
\alpha _{M_{2}\ast N}\left( n \ast f(m_1)\right) &=& f(\alpha _{N}\left( n\right) \ast \alpha _{M_{2}}\left(
m_1\right) ) & \in {\Image}\left( f\ast \id_{N}\right),
\end{array} \]
for any $m_{1}\in M_1$, $m_{2}\in M_2$, $n, n'\in N$.
Thus, $g\ast \id_{N}$ induces a homomorphism of Hom-Leibniz algebras
\[
(\left( M_{2}\ast N\right) / {\Image}\left( f\ast \id_{N}\right) ,%
\overline{\alpha }_{M_{2}\ast N}){\longrightarrow }(M_{3}\ast N,\alpha _{M_{3}\ast N}),
\]
given on generators  by $ \overline{m_{2}\ast n} \mapsto g(m_{2})\ast n$ and $ \overline{n\ast m_{2}} \mapsto n\ast g(m_{2})$, which is an isomorphism with the inverse map
\[
(M_{3}\ast N,\alpha_{M_{3}\ast N})\longrightarrow (\left( M_{2}\ast N\right) /{\Image}\left(
f\ast \id_{N}\right) ,\overline{\alpha }_{M_{2}\ast N})
\]
defined by $m_{3}\ast n\mapsto  \overline{m_{2}\ast n}$ and $n\ast m_{3}= \overline{n\ast m_{2}}$, where $m_{2}\in M_{2}$ such that $g(m_{2})=m_{3}.$
Then the required exactness follows.
\end{proof}

\begin{proposition}\label{exact-tensor-2}
 If $(M,\alpha_M)$ is a two-sided Hom-ideal of a Hom-Leibniz algebra $(L, \alpha_L)$, then there is an exact sequence of Hom-Leibniz
algebras
\[
\big((M \ast  L)\rtimes (L\ast  M), \alpha_{\rtimes}\big) \overset{\sigma} {\longrightarrow}(L\ast L,\alpha_{L\ast L})\overset{\tau}{\longrightarrow}({L}/{M}\ast {L}/{M},\alpha_{{L}/{M}\ast {L}/{M}})\longrightarrow 0.
\]
where the Hom-Leibniz action of $L\ast M $ on $M \ast L$ is naturally given by the bracket in $L\ast M$.
\end{proposition}
\begin{proof} First we note that { $\tau$} is the functorial homomorphism induced by
projection $\left( L,\alpha _{L}\right) \twoheadrightarrow \left( L/M,\alpha_{L/M}\right)$ and clearly it is an epimorphism.

 Let $\sigma ^{\prime }:(M \ast L,\alpha _{M \ast L})\to (L\ast L,\alpha _{L\ast L})$, $\sigma ^{\prime \prime }:(L\ast M,\alpha_{L\ast M})\to (L\ast L,\alpha _{L\ast L})$  be the functorial homomorphisms induced by the inclusion $(M,\alpha _{M})\hookrightarrow (L,\alpha _{L})$ and the identity map $(L,\alpha _{L})\to (L, \alpha _{L}).$
Then $\sigma : \left( \left(M \ast L \right) \rtimes \left( L \ast M\right), \alpha_{\rtimes} \right) \to (L \ast L, \alpha_{L \ast L})$, is defined by $\sigma \left( x,y\right) =\sigma ^{\prime }(x)+\alpha_{L \ast L} \circ \sigma ^{\prime \prime}(y)$ for all $x\in M \ast L$ and $y\in L \ast M$, i. e.
 \begin{align*}
 \sigma((m_1 \ast l_1), (l_2 \ast m_2))=m_1 \ast l_1 + \alpha_L(l_2) \ast \alpha_M(m_2),\\
 \sigma((l_1 \ast m_1), (l_2 \ast m_2))=l_1 \ast m_1 + \alpha_L(l_2) \ast \alpha_M(m_2),\\
 \sigma((m_1 \ast l_1), (m_2 \ast l_2))=m_1 \ast l_1 + \alpha_M(m_2) \ast \alpha_L(l_2),\\
 \sigma((l_1 \ast m_1), (m_2 \ast l_2))=l_1 \ast m_1 + \alpha_M(m_2) \ast \alpha_L(l_2),
 \end{align*}
 for all $m_1, m_2\in M$ and $l_1,l_2\in L$.

 The exactness can be checked in the same way as in the proof of Proposition \ref{exact-tensor-1} and we omit it.
 \end{proof}

\begin{proposition} \label{action-on-tensor}
Let  $(M,\alpha_M)$ and $(N,\alpha_N)$ be Hom-Leibniz algebras with compatible actions on each other.
 \begin{enumerate}
\item[a)] There are homomorphisms of Hom-Leibniz algebras
\[\begin{array}{lll}
\psi _{1}:(M\ast N,\alpha _{M\ast N})\rightarrow (M,\alpha _{M}),&
\psi _{1}(m\ast n)= m^n, & \psi_{1}(n\ast m)= {^n}m,\\
\psi _{2}:(M\ast N,\alpha _{M\ast N})\rightarrow (N,\alpha _{N}),&
\psi _{2}(m\ast n)={^m}n, & \psi_{2}(n\ast m)={n}^m.
\end{array}
\]

\item[b)]
There is a Hom-Leibniz action of  $(M, \alpha_{M })$ (resp.  $(N, \alpha_{N})$ )  on  the non-abelian tensor product $(M\ast N, \alpha_{M \ast N})$  given, for all $m,m'\in M$, $n,n'\in N$, by
\[\begin{array}{rcl}
{}^{m^{\prime }}(m\ast n)&=&{}[m^{\prime },m]\ast \alpha
_{N}(n)-^{m^{\prime }}n\ast \alpha _{M}(m), \\
{}^{m^{\prime }}(n\ast m)&=& ^{m^{\prime }}n\ast \alpha
_{M}(m)-[m^{\prime },m]\ast \alpha _{N}(n),\\
(m\ast n)^{m^{\prime }}&=&[m,m^{\prime }]\ast \alpha _{N}(n)+\alpha
_{M}(m)\ast {}n^{m^{\prime }}, \\
(n\ast m)^{m^{\prime }}&=&n^{m^{\prime }}\ast \alpha _{M}(m)+\alpha
_{N}(n)\ast \lbrack m,m^{\prime }],

\end{array}\]

\[
\begin{array}{rcl}
\big(\ \text{resp.} \
^{n^{\prime }}(m\ast n)&=& ^{n^{\prime }}m\ast \alpha
_{N}(n)-\left[n^{\prime },n\right] \ast \alpha _{M}(m), \\
^{n^{\prime }}(n\ast m) &=& \left[ n^{\prime },n\right] \ast \alpha
_{M}(m)-^{n}m\ast \alpha _{N}(n^{\prime }),\\
 (m\ast n)^{n^{\prime }}&=& m^{n^{\prime }}\ast \alpha _{N}(n)+\alpha
_{M}(m)\ast \left[ n,n^{\prime }\right], \\
{}(n\ast m)^{n^{\prime }}&=&\left[ n,n^{\prime }\right] \ast \alpha
_{M}(m)+\alpha _{N}(n)\ast m^{n^{\prime }}\ \big).
\end{array}
\]

\item[d)] $\Ker(\psi _{1})$ and $\Ker(\psi _{2})$ both are contained in the center of $(M\ast N, \alpha_{M \ast N})$.

\item[e)] The induced Hom-Leibniz action of $\Image(\psi _{1})$ (resp. $\Image(\psi _{2})$) on  $\Ker(\psi_{1})$ (resp. $\Ker(\psi _{2})$) is trivial.

        \item[f)]  $\psi_1$ and $\psi_2$ satisfy the following properties for all $m, m' \in M$, $n, n' \in N$:
\begin{enumerate}
\item[i)] $\psi_1(^{m'}(m \ast n)) = [\alpha_M(m'), \psi_1(m \ast n)]$,
\item[ii)] $\psi_1((m \ast n)^{m'}) = [\psi_1(m \ast n),\alpha_M(m')]$,
\item[iii)] $\psi_1(^{m'}(n \ast m)) = [\alpha_M(m'), \psi_1(n \ast m)]$,
\item[iv)] $\psi_1((n \ast m)^{m'}) = [ \psi_1(n \ast m), \alpha_M(m')]$,
\item[v)] $\psi_2(^{n'}(m \ast n)) = [\alpha_N(n'), \psi_2(m \ast n)]$,
\item[vi)] $\psi_2((m \ast n)^{n'}) = [\psi_2(m \ast n),\alpha_N(n')]$,
\item[vii)] $\psi_2(^{n'}(n \ast m)) = [\alpha_N(n'), \psi_2(n \ast m)]$,
\item[viii)] $\psi_2((n \ast m)^{n'}) = [ \psi_2(n \ast m), \alpha_N(n')]$,
\item[ix)] ${^{\psi_1(m \ast n)}}(m' \ast n') = [\alpha_{M \ast N}(m \ast n), m' \ast n'] = {^{\psi_2(m \ast n)}} (m' \ast n')$,
\item[x)] ${^{\psi_1(m \ast n)}}(n' \ast m') = [\alpha_{M \ast N}(m \ast n), n' \ast m'] = {^{\psi_2(m \ast n)}} (n' \ast m')$,
 \item[xi)] ${^{\psi_1(n \ast m)}}(m' \ast n') = [\alpha_{M \ast N}(n \ast m), m' \ast n'] = {^{\psi_2(n \ast m)}} (m' \ast n')$,
\item[xii)] ${^{\psi_1(n \ast m)}}(n' \ast m') = [\alpha_{M \ast N}(n \ast m), n' \ast m'] = {^{\psi_2(n \ast m)}} (n' \ast m')$,
 \item[xiii)] $(m' \ast n')^{\psi_1(m \ast n)} = [ m' \ast n',\alpha_{M \ast N}(m \ast n)] =  (m' \ast n')^{\psi_2(m \ast n)}$,
 \item[xiv)] $(n' \ast m')^{\psi_1(m \ast n)} = [ n' \ast m',\alpha_{M \ast N}(m \ast n)] =  (n' \ast m')^{\psi_2(m \ast n)}$,   \item[xv)] $(m' \ast n')^{\psi_1(n \ast m)} = [ m' \ast n',\alpha_{M \ast N}(n \ast m)] =  (m' \ast n')^{\psi_2(n \ast m)}$,
     \item[xvi)] $(n' \ast m')^{\psi_1(n \ast m)} = [ n' \ast m',\alpha_{M \ast N}(n \ast m)] =  (n' \ast m')^{\psi_2(n \ast m)}$.
    \end{enumerate}

\end{enumerate}
\end{proposition}
\begin{proof}
 The proof requires routine combinations of equations in Definitions \ref{Hom accion} and \ref{compatibility1} with the defining relations (\ref{a)}) and (\ref{bracket}) of the non-abelian Hom-Leibniz tensor product.
\end{proof}

\begin{remark}\label{remark} \
\begin{enumerate}
\item[a)] If $\alpha_M=\id_M$ and $\alpha_N=\id_N$ then the statement f) of Lemma \ref{action-on-tensor} says that both $\psi_1$ and $\psi_2$ are crossed modules of Leibniz algebras (see \cite{Gn}).
\item[b)] If $M=N$ and the action of $M$ on itself is given by the bracket in $M$, then $\psi_1=\psi_2$ and they are defined on generators by $m\ast m'\mapsto [m,m']$. In such a case we write $\psi_M$ for $\psi_1$.
\end{enumerate}
\end{remark}


\section{ Application in universal ($\alpha$-)central extensions of Hom-Leibniz algebras}\label{section4}

In this section we complement by new results the investigation of universal central extensions of Hom-Leibniz algebras done in \cite{CIP1}. We also describe universal ($\alpha$)-central extensions via non-abelian Hom-Leibniz tensor product.

\begin{definition} \label{alfacentral} An epimorphism of Hom-Leibniz algebras
  $(K,\alpha_K) \stackrel{\pi}{\twoheadrightarrow} (L, \alpha_L)$, with $\Ker(\pi)=(M,\alpha_M)$, is called a central (resp. $\alpha$-central) extension of $(L, \alpha_L)$ if $[M, K] = 0 = [K,M]$, i.e. $M \subseteq Z(K)$ (resp. $[\alpha_M(M), K] = 0 = [K, \alpha_M(M)]$, i.e. $\alpha_M(M) \subseteq Z(K)$).

Such a central extension is called universal central (resp. $\alpha$-central) extension if, for every central (resp. $\alpha$-central) extension $(K',\alpha_{K'}) \stackrel{\pi} {\twoheadrightarrow} (L, \alpha_L)$ of $(L,\alpha_L)$ there exists one and only one homomorphism of Hom-Leibniz algebras $h : (K,\alpha_K) \to (K',\alpha_{K'})$ such that $\pi' \circ h = \pi$.
\end{definition}

\begin{remark}
Obviously every  central extension is an $\alpha$-central extension and these notions coincide when $\alpha_M = \id_M$. On the other hand, every universal $\alpha$-central extension  is a  universal  central extension and these notions coincide when $\alpha_M = \id_M$. Let us also observe that if a universal ($\alpha$-)central  extension exists then it is unique up to isomorphism.
\end{remark}

The category ${\sf HomLb}$ is an example of a semi-abelian category which does not satisfy universal central extension condition in the sense of \cite{CVdL}, that is, the composition of central extensions of Hom-Leibniz algebras is not central in general, but it is $\alpha$-central extension (see Theorem \ref{teorema} (a) below). This fact does not allow complete generalization of classical results to Hom-Leibniz algebras and the well-known properties of universal central extensions are divided between universal central and universal $\alpha$-central extensions of Hom-Leibniz algebras. In particular, the assertions in Theorem \ref{teorema} below are proved in \cite{CIP1}.

To state the following theorem we need to recall that a Hom-Lie algebra  $(L, \alpha_L)$ is said to be \emph{perfect} if $L=[L, L]$.

\begin{theorem}\label{teorema} \
\begin{enumerate}
       \item[a)] Let $(K,\alpha_K) \stackrel{\pi} \twoheadrightarrow (L, \alpha_L)$  and  $(F,\alpha_F) \stackrel{\rho} \twoheadrightarrow (K, \alpha_K)$ be  central extensions of Hom-Leibniz algebras and $(K, \alpha_K)$ be a perfect Hom-Leibniz algebra. Then the composition extension $(F,\alpha_F) \stackrel{\pi \circ \rho} \twoheadrightarrow (L, \alpha_L)$ is an $\alpha$-central extension.

    \item[b)] Let  $(K,\alpha_K) \stackrel{\pi} \twoheadrightarrow (L, \alpha_L)$ and $(K',\alpha_{K'}) \stackrel{\pi'} \twoheadrightarrow (L, \alpha_L)$ be two  central extensions of a Hom-Leibniz algebra $(L,\alpha_L)$. If $(K,\alpha_K)$ is perfect, then there exists at most one homomorphism of Hom-Leibniz algebras  $f : (K,\alpha_K) \to (K', \alpha_{K'})$ such that $\pi' \circ f = \pi$.

 \item[c)] Let  $(K,\alpha_K) \stackrel{\pi} \twoheadrightarrow (L, \alpha_L)$ be a central extension and $(K',\alpha_{K'}) \stackrel{\pi'} \twoheadrightarrow (L, \alpha_L)$ be an $\alpha$-central extension of a Hom-Leibniz algebra $(L,\alpha_L)$. If $(K,\alpha_K)$ is $\alpha$-perfect, then there exists at most one homomorphism of Hom-Leibniz algebras  $f : (K,\alpha_K) \to (K', \alpha_{K'})$ such that $\pi' \circ f = \pi$.

\item[d)]  If $ (K,\alpha_K) \stackrel{\pi} \twoheadrightarrow (L, \alpha_L)$ is a universal $\alpha$-central extension, then  $(K,\alpha_K)$ is a perfect   Hom-Leibniz algebra and every central extension of $(K,\alpha_K)$ splits.

    \item[e)] If $(K,\alpha_K)$ is a perfect  Hom-Leibniz algebra and every  central extension of $(K,\alpha_K)$ splits, then any central extension $(K,\alpha_K) \stackrel{\pi} \twoheadrightarrow (L, \alpha_L)$ is a universal central extension.

\item[f)] A Hom-Leibniz algebra $(L, \alpha_L)$ admits a universal central extension if and only if $(L, \alpha_L)$ is perfect.
Furthermore, the kernel of the universal central extension is canonically isomorphic to the second homology $HL_2^{\alpha}(L)$.

\item[g)] An $\alpha$-perfect Hom-Leibniz algebra admits a universal $\alpha$-central extension.

\item[h)] If $(K,\alpha_K) \stackrel{\pi} \twoheadrightarrow (L, \alpha_L)$ is a universal $\alpha$-central extension, then $HL_1^{\alpha}(K) = HL_2^{\alpha}(K) = 0$.

   \item[i)]   If $HL_1^{\alpha}(K) = HL_2^{\alpha}(K) = 0$, then any central extension  $(K,\alpha_K) \stackrel{\pi} \twoheadrightarrow (L, \alpha_L)$  is a universal  central extension.

\end{enumerate}
\end{theorem}

Given a Hom-Leibniz algebra $(L,\alpha_L)$ we have the epimorphism of  Hom-Leibniz algebras
\[
 \psi_L:(L\ast L, \alpha_{L \ast L}) \twoheadrightarrow ([L,L], \alpha_{L \mid}), \quad  \psi_L(l \ast l')=[l,l'],
 \]
(cf. Remark \ref{remark} b)) which is a  central extension of $([L,L], \alpha_{L \mid})$ because of the equalities (\ref{bracket}). Moreover, we have

\begin{theorem} \label{teor}
If $(L,\alpha_L)$ is a perfect Hom-Leibniz algebra,  then the central extension $(L\ast L, \alpha_{L \ast L})\overset{\psi_L}\twoheadrightarrow (L, \alpha_L)$ is the universal central extension of $(L,\alpha_L)$.
\end{theorem}
\begin{proof} Let $(C, \alpha_C) \overset{\phi}\twoheadrightarrow (L, \alpha_L)$ be a central extension of $(L, \alpha_L)$. Since $\Ker\left(\phi\right) \subseteq Z\left(  C \right)$, we get a well-defined homomorphism of Hom-Leibniz algebras $f:(L\ast L,$ $ \alpha_{L\ast L})\to (C, \alpha_C)$ given on generators by $f(l\ast l')=[c_l,c_{l'}]$, where $c_l$ and $c_{l'}$ are any elements in $\phi^{-1}(l)$ and $\phi^{-1}(l')$, respectively. Obviously $\phi \circ f= \psi_L$ and $f \circ \alpha_{L \ast L} = \alpha_C \circ f$, having in mind that $\alpha_C(c_l) \in \phi^{-1}(\alpha_L(l))$ for all $l \in L$.  Since $L$ is perfect, then by equalities (\ref{bracket}), so is $L\ast L$. Hence the homomorphism $f$ is unique by Theorem \ref{teorema} {\it b)}.
\end{proof}

\begin{remark} \label{H2}
If the Hom-Leibniz algebra $(L,\alpha_L)$ is perfect, by Theorem \ref{teorema} {e)} we have $HL_2^{\alpha}(L)\approx  \Ker(L\ast L\overset{\psi_L}\to L )$.
\end{remark}

\begin{theorem} \label{sucesion exacta}
Let $(M,\alpha_M)$ be a two-sided Hom-ideal of a perfect Hom-Leibniz algebra $(L,\alpha_L)$. Then there is an exact sequence of vector spaces
 \[
 \Ker(L\ast M\overset {\psi_{2}\ }\longrightarrow L)\to HL_2^{\alpha}(L)\to HL_2^{\alpha}({L}/{M})\to {M}/{[L,M]}\to  0.
 \]
\end{theorem}
\begin{proof} By Proposition \ref{exact-tensor-2} there is a commutative diagram of Hom-Leibniz algebras with exact rows
{\footnotesize\[
 \xymatrix{
  & \big((M \ast  L)\rtimes (L\ast  M), \alpha_{\rtimes}\big)  \ar[r]^{\ \ \ \ \ \ \ \ \sigma} \ar[d]^{\psi}& (L \ast L,\alpha_{L \ast L}) \ar[r]^{\tau \ \ \ \ \ \ \ \ } \ar[d]^{\psi_L} & ({L}/{M} \ast {L}/{M}, {\alpha}_{{L}/{M} \ast {L}/{M}}) \ar[r]  \ar[d]^{{\psi_{L/M}}}& 0\\
0 \ar[r] & (M, \alpha_M) \ar[r] & (L,\alpha_L) \ar[r]^{\pi} & ({L}/{M},{\alpha}_{L/M}) \ar[r] & 0
}
\]}

\noindent where
\[
\begin{array}{crcl}
&\psi((m_1\ast l_1),(l_2\ast m_2))&=&[m_1,l_1]+[\alpha_L(l_2),\alpha_M(m_2)],\\
&\psi((l_1\ast m_1),(l_2\ast m_2))&=&[l_1,m_1]+[\alpha_L(l_2),\alpha_M(m_2)],\\
&\psi((m_1\ast l_1),(m_2\ast l_2))&=&[m_1,l_1]+[\alpha_M(m_2),\alpha_L(l_2)],\\
&\psi((l_1\ast m_1),(m_2\ast l_2))&=&[l_1,m_1]+[\alpha_M(m_2),\alpha_L(l_2)],
\end{array}\]
and
\[
\begin{array}{crcl}
&\psi_L(l_1\ast l_2)&=&[l_1,l_2],\\
&\psi_{L/M}(\overline{l_1}\ast\overline{l_2})&=&\overline{[l_1,l_2]},
\end{array}
\]
 for all $m_1, m_2\in M$, $l_1, l_2\in L$. Then, forgetting the Hom-Leibniz algebra structures, by using the Snake Lemma for the same diagram of vector spaces,  the assertion follows from  Remark \ref{H2} and the fact that there is a surjective map $\Ker(\psi) \to \Ker(\psi_2)$.
\end{proof}

Now suppose $\left( L, \alpha_{L}\right)$ is a perfect Hom-Lie algebra. By \cite[Theorem 3.4 and Remark 3.5]{CKP} we know that its universal central extension in the category {\sf Hom-Lie} is
\begin{equation}\label{eq}
(L\star L, \alpha_{L\star L})\overset{u_L}\twoheadrightarrow ( L, \alpha_{L}),
 \end{equation}
 and $\Ker (u_L)\approx H_{2}^{\alpha}(L)$, where $\star$ denotes the non-abelian Hom-Lie tensor product and $H_{2}^{\alpha}(L)$ is the second homology { with trivial coefficients} of the Hom-Lie algebra $\left( L, \alpha_{L}\right)$. Furthermore, since any perfect Hom-Lie algebra is a perfect Hom-Leibniz algebra as well, by Theorem \ref{teorema}  {\it f)}, $\left(  L, \alpha_{L}\right)$ admits  universal central extension in the category {\sf HomLb}, which is described via the non-abelian Hom-Leibniz tensor product in Theorem \ref{teor} and Remark \ref{H2}. Then,
  by the universal property of $(L\ast L,\alpha_{L\ast L})\overset{\psi_L}\twoheadrightarrow ( L, \alpha_{L})$ and
  considering (\ref{eq}) as a central extension of Hom-Leibniz algebras, we get a homomorphism of Hom-Leibniz algebras $\xi:(L\ast L, \alpha_{L\ast L})\to (L\star L, \alpha_{L\star L})$, $\xi({l\ast l'})={l\star l'}$, such that $\psi_L=u_L\circ \xi$. Note that $\xi$ is the epimorphism  described in Remark \ref{remarkLieLb} c).  Now the proof of the following proposition is obvious.
  \begin{proposition}
  Let $\left( L, \alpha_{L}\right)$ be a perfect Hom-Lie algebra, then the non-abelian Hom-Lie and Hom-Leibniz tensor squares  are isomorphic, $\left(  L \star L, \alpha_{L \star L} \right)\approx \left(  L \ast L,\right.$ $\left. \alpha_{L \ast L} \right)$, if and only if the second Hom-Lie and Hom-Leibniz homologies { with trivial coefficients} are isomorphic, $H_{2}^{\alpha}\left(  L\right) \approx HL_{2}^{\alpha}\left(  L\right)$.
  \end{proposition}

Now we deal with the universal $\alpha$-central extension of { an $\alpha$-perfect} Hom-Leibniz algebra. We need the following notion.

 \begin{definition} \cite{CIP1}
 A Hom-Leibniz algebra $\left(  L, \alpha_{L}\right)$ is said to be $\alpha$-perfect if
 $$L = [\alpha_L(L), \alpha_L(L)]$$
 \end{definition}


 \begin{remark}\label{alfa perfecta} \
 \begin{enumerate}
 \item[a)] When $\alpha_L= \id$ the notions of perfect and $\alpha$-perfect Hom-Leibniz algebras are the same.
\item[b)] Any $\alpha$-perfect  Hom-Lie algebra is an $\alpha$-perfect  Hom-Leibniz algebra.
 \item[c)] Any $\alpha$-perfect Hom-Leibniz algebra is perfect. Nevertheless the converse is not true in general.  An example of this is given in \cite[Remark 3.8 b)]{CKP}.
 \item[d)] If  $\left(L, \alpha_{L}\right)$  is $\alpha$-perfect, then $L = \alpha_L(L)$, i.e. $\alpha_L$ is surjective.
 \end{enumerate}
 \end{remark}

It is shown in \cite[Theorem 5.5 (b)]{CIP1} that an $\alpha$-perfect Hom-Leibniz algebra admits a universal $\alpha$-central extension. Now we give  its description via the non-abelian Hom-Leibniz tensor product.

\begin{theorem}\label{alfa uce} \
 Let $\left(  L, \alpha_{L}\right)$ be an $\alpha$-perfect Hom-Leibniz algebra.  Then  the epimorphism $\psi : \left(  \alpha_L(L) \ast \alpha_L(L), \alpha_{\alpha_L(L) \ast \alpha_L(L)} \right) \twoheadrightarrow \left(  L, \alpha_{L}\right)$, $\psi \left(  \alpha_{L}\left(  l\right) \ast\alpha_{L}\left(  l^{\prime}\right)  \right)  =\left[  \alpha_{L}\left( {l}\right)  ,\alpha_{L}\left(  {l^{\prime}}\right)  \right] $,  is its universal $\alpha$-central extension.
\end{theorem}
\begin{proof}
 Obviously $\psi$ is a central extension. For every $\alpha$-central extension $\left(  C,\alpha_{C}\right)  \overset{\phi}{\twoheadrightarrow}\left(  L,\alpha_{L}\right)$
there is a homomorphism $f:\left(  \alpha_L(L) \ast \alpha_L(L), \alpha_{\alpha_L(L) \ast \alpha_L(L)} \right)  \to \left(  C,\alpha_{C}\right) $ given on generators by $f\left(  \alpha_{L}\left(  l\right) \ast\alpha_{L}\left(  l^{\prime}\right)  \right)  =\left[  \alpha_{C}\left( c_{l}\right)  ,\alpha_{C}\left(  c_{l^{\prime}}\right)  \right] $, where $c_l$ and $c_{l'}$ are any elements in $\phi^{-1}(l)$ and $\phi^{-1}(l')$, respectively. $f$ is well-defined since  $\alpha_{C}\left(  \Ker\left(  \phi\right)  \right)  \subseteq Z\left(  C \right)$.
Obviously $\phi \circ f= \psi$ and $f \circ \alpha_{\alpha_L(L) \ast \alpha_L(L)} = \alpha_C \circ f$, having in mind that $\alpha_C(c_l) \in \phi^{-1}(\alpha_L(l))$ and $\alpha_C^2(c_l) \in \phi^{-1}(\alpha_L^2(l))$ for all $l \in L$.  Since $L$ is $\alpha$-perfect, then by equality (\ref{bracket}), so is $ \alpha_{L}(L)  \ast\alpha_{L}(L)$. Hence the homomorphism $f$ is unique by Theorem \ref{teorema} {\it c)}.
\end{proof}

An alternative construction of the universal $\alpha$-central extension $(\frak{uce}_{\alpha}^{\sf{Lb}}(L), \overline{\alpha})$ of an $\alpha$-perfect Hom-Leibniz algebra $\left(  L, \alpha_{L}\right)$ is given in  \cite[Theorem 5.5 (b)]{CIP1}. Let us recall that $\frak{uce}_{\alpha}^{\sf{Lb}}(L)$ is the quotient of ${\alpha_L(L) \otimes \alpha_L(L)}$ by the vector subspace spanned by all elements $-[x_1,x_2] \otimes \alpha_L(x_3) + [x_1,x_3] \otimes \alpha_L(x_2) +  \alpha_L(x_1) \otimes [x_2,x_3]$, where $x_1, x_2, x_3 \in L$, the bracket of $\frak{uce}_{\alpha}^{\sf{Lb}}(L)$ is given on generators by
\[
[\alpha_L(x_1)\otimes \alpha_L(x_2), \alpha_L(y_1)\otimes \alpha_L(y_2)]=[\alpha_L(x_1), \alpha_L(x_2)]\otimes [\alpha_L(y_1), \alpha_L(y_2)],
\]
$x_1,x_2,y_1,y_2\in L$,
 and $\overline{\alpha}:\frak{uce}_{\alpha}^{\sf{Lb}}(L)\to \frak{uce}_{\alpha}^{\sf{Lb}}(L)$ is induced by $\alpha_L$.

Since the universal $\alpha$-central extension of an $\alpha$-perfect Hom-Leibniz algebra $\left(  L, \alpha_{L}\right)$  is unique up to isomorphisms, then there is an isomorphism of Hom-Leibniz algebras
 \[
 \left(  \alpha(L) \ast \alpha(L), \alpha_{\alpha(L) \ast \alpha(L)} \right) \approx (\frak{uce}_{\alpha}^{\sf{Lb}}(L), \overline{\alpha}).
 \]

Now suppose that $\left(  L,\alpha_{L}\right)$  is an $\alpha$-perfect Hom-Lie algebra. By \cite[Theorem 5.5 (a)]{CIP1}  it admits a universal $\alpha$-central extension in the category {\sf{HomLie}} of Hom-Lie algebras, which is described as the non-abelian Hom-Lie tensor product $\left(  \alpha(L) \star \alpha(L), \alpha_{\alpha(L) \star \alpha(L)} \right)$ in \cite[Theorem 3.11]{CKP}. $\left(  L,\alpha_{L}\right)$ admits a universal $\alpha$-central extension in the category {\sf{HomLb}} described in Theorem \ref{alfa uce} via non-abelian Hom-Leibniz tensor product.
Then \cite[Proposition 5.6]{CIP1} provides the following relationship between the non-abelian Hom-Lie and Hom-Leibniz tensor products:
\[
\left(  \alpha(L) \star \alpha(L), \alpha_{\alpha(L) \star \alpha(L)} \right) \approx  \left( \left( \alpha(L) \ast \alpha(L)\right)_{\sf{Lie}}, \overline{\alpha}_{\alpha(L) \ast \alpha(L)} \right).
\]
 Moreover, $\left( \alpha(L) \star \alpha(L), \alpha_{\alpha(L) \star \alpha(L)}\right)$ is an $\alpha$-perfect Hom-Lie algebra and  its universal $\alpha$-central extension  in the category  {\sf{HomLb}} of Hom-Leibniz algebras is isomorphic to  $\left(  \alpha(L) \ast \alpha(L), \alpha_{\alpha(L) \ast \alpha(L)} \right)$.


\section{Application in Hochschild homology of Hom-associative algebras }\label{section5}

\begin{definition}
A (multiplicative) Hom-associative algebra $(A, \alpha_A)$ is a vector space $A$ together with linear maps $\alpha_A:A\to A$ and $A\otimes A\to A$, $a\otimes b \mapsto ab$, such that
\begin{align*}
&\alpha_A(a)(bc)=(ab)\alpha_A(c),\\
&\alpha_A(ab)=\alpha_A(a)\alpha_A(b)
\end{align*}
for all $a,b,c\in A$.
\end{definition}

 Any Hom-associative algebra $(A,\alpha _{A})$ can be viewed as a Hom-Leibniz (in fact Hom-Lie) algebra  with the bracket defined by $\left[ a,b\right] :=ab-ba$, $a,b\in A$ (see Example 5.2 in \cite{Yau1}).

 Given a Hom-associative algebra $(A,\alpha _{A})$, we denote $\mathbb{L}^{\alpha}(A):=(A{\otimes}A)/\Image(b_3)$, where $b_{3}:A\otimes A\otimes A\to A\otimes A$, $b_{3}\left( a\otimes b\otimes c\right) =ab\otimes \alpha _{A}(c)-\alpha_{A}(a)\otimes bc+ca\otimes \alpha _{A}(b)$, is the boundary map for the Hochschild complex of the Hom-associative algebra $(A,\alpha _{A})$ constructed in \cite{Yau1}. Then  we have the following short exact sequence of
vector spaces
  \begin{equation}\label{eq3}
0 \longrightarrow HH_1^{\alpha}(A) \stackrel{i}\longrightarrow \mathbb{L}^{\alpha }(A) \stackrel{\phi}\longrightarrow [A,A] \longrightarrow 0\ ,
  \end{equation}
  where  $HH_1^{\alpha}(A)$ denotes the first Hochschild homology of $(A,\alpha _{A})$ \cite{Yau1},  $[A,A]$ is the subspace of $A$ spanned by the elements $ab-ba$  and $\phi(a \boxtimes b) = ab - ba$, $a,b\in A$. Here $a\boxtimes b$ stands for the equivalence class of $a\otimes b$.

Suppose $\overline{\alpha}_{A}:\mathbb{L}^{\alpha}(A)\to\mathbb{L}^{\alpha}(A)$ is the linear map induced by $\alpha_A$, i.e. $\overline{\alpha}_{A}({a \boxtimes b}) = {\alpha_A(a) \boxtimes \alpha_A(b)}$. Then there is a Hom-Leibniz algebra structure on $(\mathbb{L}^{\alpha}(A), \overline{\alpha}_{A})$
 given by the bracket
\[
\left[ {a\boxtimes b},{a^{\prime }\boxtimes b^{\prime }}\right] :={\left( ab-ba\right) \boxtimes \left( a^{\prime }b^{\prime}-b^{\prime }a^{\prime }\right)} = {\left[ a,b\right]\boxtimes \left[ a^{\prime },b^{\prime }\right]},
\]
for all $a,b,a^{\prime },b^{\prime }\in A$. Indeed, it is immediate that
\begin{equation}\label{eq2}
[a,b]\boxtimes \alpha _{A}(c)-\alpha_{A}(a)\boxtimes [b,c]+[c,a]\boxtimes \alpha _{A}(b)=0,
\end{equation}
 then the Hom-Leibniz identity (\ref{def}) easily follows.

\begin{remark}\label{remarkNote}
 $\left( \mathbb{L}^{\alpha }(A),\overline{\alpha}_{A}\right)$ is the quotient of the non-abelian Hom-Leibniz tensor product $\left( A\ast A,\alpha _{A\ast A}\right)$ by the two-sided Hom-ideal generated by the elements of the form $ab\ast \alpha _{A}(c)-\alpha _{A}(a)\ast bc+ca\ast \alpha _{A}(b),$ for all $a,b,c\in A.$
\end{remark}

\begin{definition} \cite{CKP}
We say that a Hom-associative algebra $(A, \alpha_A)$ satisfies the $\alpha$-identity condition if
\begin{equation} \label{alpha identity}
[A, \Image(\alpha_A - \id_A)]=0\ ,
\end{equation}
where $[A, \Image(\alpha_A - \id_A)]=0$ is the subspace of $A$ spanned by all elements $ab-ba$, with $a \in A$ and $b \in  \Image(\alpha_A - \id_A)$.
\end{definition}

Examples of Hom-associative algebras satisfying the $\alpha$-identity condition can be found in \cite{CKP}.

\begin{proposition} \label{main} Let $(A,\alpha _{A})$ be a Hom-associative algebra.
\begin{enumerate}
\item[a)] There is a Hom-Leibniz action of $(A,\alpha _{A})$ on $\left( \mathbb{L}^{\alpha }(A),\overline{\alpha}_{A}\right)$ given, for all $a^{\prime }\in A$, $a\boxtimes b\in \mathbb{L}^{\alpha }(A),$ by
\begin{align*}
&{}^{a^{\prime }}\left( a\boxtimes b\right) =\left[ a^{\prime },a\right] \boxtimes \alpha _{A}\left( b\right) -\left[ a^{\prime },b
\right] \boxtimes \alpha _{A}\left( a\right),\\
&\left( a\boxtimes b\right) ^{a^{\prime }}=\left[ a,a^{\prime } \right] \boxtimes \alpha _{A}\left( b\right) +\alpha _{A}\left( a\right) \boxtimes \left[ b,a^{\prime }\right],
\end{align*}
and a Hom-Leibniz action of $\left( \mathbb{L}^{\alpha }(A),\overline{\alpha}_{A}\right)$  on $(A,\alpha _{A})$ given by
\begin{align*}
&{}^{\left( {a\boxtimes b}\right) }a^{\prime }=\left[ \left[a,b\right] ,a^{\prime }\right],\\
&a^{\prime \left( {a\boxtimes b}\right) }=\left[ a^{\prime} ,\left[ a,b\right] \right].
\end{align*}

Moreover, these actions are compatible if $(A,\alpha _{A})$  satisfies the $\alpha$-identity condition (\ref{alpha identity}).

\item[b)] There is a short exact sequence of Hom-Leibniz algebras
\[
0 \to (HH_1^{\alpha}(A), \overline{\alpha}_{A \mid}) \stackrel{i}\to \left( \mathbb{L}^{\alpha }(A),\overline{\alpha}_{A}\right) \stackrel{\phi} \to ([A,A],\alpha_{A \mid}) \to 0
\]
where $(HH_1^{\alpha}(A), \alpha_{A \mid})$ is considered as an abelian Hom-Leibniz algebra, $\overline{\alpha}_{A \mid}$ (resp. $\alpha_{A \mid}$) is the restriction of $\overline{\alpha} _{A}$ (resp. $\alpha_A$) and $\phi(a \boxtimes b) =[a,b]=ab - ba$.

\item[c)] The induced Hom-Leibniz action of $(A,\alpha _{A})$  on $(HH_1^{\alpha}(A), \alpha_{A \mid})$ is trivial. Moreover, $i$ always preserves the actions and, if $(A,\alpha _{A})$ satisfies the $\alpha$-identity condition (\ref{alpha identity}), then $\phi$ also preserves the actions.

\end{enumerate}
\end{proposition}
\begin{proof}
 { a)}  By using formula (\ref{eq2}) and Remark \ref{remarkNote}, one can readily check that the actions are well-defined Hom-Leibniz actions. It is also easy to see that these actions are compatible under condition (\ref{alpha identity}). For example, we have
 \begin{align*}
 {}^{({}^{a'}(a\boxtimes b))}{a''}&= {}^{[a',a]\boxtimes \alpha_A(b)}a''-{}^{[a',b]\boxtimes \alpha_A(a)}a'' = \big[[[a',a],\alpha_A(b)]-[[a',b],\alpha_A((a))],a''\big]\\
 & \overset{(\ref{def})}=\big[[\alpha_A(a'),[a,b]], a''\big] \overset{(\ref{alpha identity})}= \big[[a',[a,b]], a''\big] = [a'^{\ a\boxtimes b}, a''].
 \end{align*}

{\it b)} It is an immediate consequence of (\ref{eq3}) and the definitions above.

{\it c)}  Since ${}^{a^{\prime }}\left( a\boxtimes b\right) =\alpha_A(a^{\prime }) \boxtimes [a, b] =0$ when $a\boxtimes b\in \Ker(\phi) $, we deduce that
the action of $(A,\alpha _{A})$ on $(HH_1^{\alpha}(A), \overline{\alpha}_{A \mid})$ is trivial and, by this reason,  $i$ preserves the actions of $(A,\alpha _{A})$. Now, if  $(A,\alpha _{A})$ satisfies the $\alpha$-identity condition (\ref{alpha identity}), then we have
\begin{align*}
&\phi\big({}^{a'}(a\boxtimes b) \big)=\big[\alpha_A(a'),[a,b] \big]\overset{(\ref{alpha identity})}{=}\big[a',[a,b] \big]={}^{a'}\phi(a\boxtimes b),\\
&\phi\big((a\boxtimes b){}^{a'} \big)=\big[[a,b],\alpha_A(a') \big]\overset{(\ref{alpha identity})}{=}\big[[a,b],a' \big]=\phi(a\boxtimes b){}^{a'},
\end{align*}
for all $a',a, b\in A$, implying that $\phi$  also preserves the actions.
\end{proof}

\begin{remark}
The statements in {Proposition} \ref{main} when $\alpha_A = \id_A$ recover the corresponding results established in \cite[Section 7]{Gn} for Leibniz algebras.
\end{remark}

In the following definition we introduce the Hom-version of the first Milnor-type Hochschield homology of associative algebras (see e.g. \cite{Gn}).

\begin{definition}
Let $(A,\alpha_A)$ be a Hom-associative algebra. The first Milnor-type Hochschield homology  $HH_1^M(A,\alpha_A)$ is the quotient of the  vector space  $A\otimes A$ by the relations
\begin{align*}
 ab\otimes \alpha_A(c)- \alpha_A(a)\otimes bc+ca\otimes \alpha_A(b) &=0,\\
\alpha_A(a)\otimes bc -\alpha_A(a)\otimes cb &=0,\\
ab \otimes \alpha_A(c) - ba \otimes \alpha_A(c) &=0,
\end{align*}
for any $a, b, c \in A$.
\end{definition}

\begin{theorem}\label{application}
Let $(A,\alpha_A)$ be a Hom-associative (non-commutative) algebra satisfying the $\alpha$-identity condition (\ref{alpha identity}). Then there is an exact sequence of vector spaces
\begin{align*}
A\ast HH_1^{\alpha}(A) &\to  \Ker(A \ast \mathbb{L}^{\alpha}(A) \to \mathbb{L}^{\alpha}(A)) \to
\Ker(A \ast[A,A] \to [A,A]) \\
&\to   HH^{\alpha}_1(A)\to HH_1^M(A,\alpha_A) \to {[A,A]}/{[A,[A,A]]}\to 0.
\end{align*}
\end{theorem}
\begin{proof}
By using { Proposition \ref{main}} and Proposition \ref{exact-tensor-1} we have the commutative diagram of Hom-Leibniz algebras (written without $\alpha$ endomorphisms)
\[
 \xymatrix{
  & A\ast HH_1^{\alpha}(A) \ar[r]  \ar[d]^{\psi_{2\mid}}& A \ast \mathbb{L}^{\alpha}(A) \ar[r] \ar[d]^{\psi_2} & A\ast [A,A] \ar[r]  \ar[d]^{{{\psi}_3}}& 0\\
0 \ar[r] & HH_1^{\alpha}(A) \ar[r] & \mathbb{L}^{\alpha}(A) \ar[r]^{\phi} & [A,A] \ar[r] & 0,
}
\]
where $\psi$ homomorphisms are defined as in { Proposition} \ref{action-on-tensor}.
 $\psi_2(a \ast (x \otimes y))= {^a}(x \otimes y)$, $\psi_2((x \otimes y) \ast a )= (x \otimes y){^a}$ and ${{\psi}_3}(a \ast [x, y])= [\alpha_A(a), [x, y]]$, $\overline{\psi}_2((x \otimes y) \ast a )= [[x, y], \alpha_A(a)]$.

Thus  $\Coker({\psi}_3)= [A,A]/ [A,[A,A]]$, $\Coker(\psi_2)=HH_1^M(A,\alpha_A)$, $\Coker({\psi_{2 \mid}})=HH_1^{\alpha}(A)$ and $\Ker(\psi_{2 \mid})=A\ast HH_1^{\alpha}(A)$. Then the assertion is a consequence of the Snake Lemma.
\end{proof}

\begin{remark}\
\begin{enumerate}
\item[a)]Let $(A,\alpha_A)$ be a Hom-associative (non-commutative) algebra satisfying the $\alpha$-identity condition (\ref{alpha identity}). If in addition $\alpha_A$ is an epimorphism, then the term $A\ast HH_1^{\alpha}(A)$ in the exact sequence of Theorem \ref{application} can be replaced by $\big(A/[A,A]\otimes HH_1^{\alpha}(A)\big) \oplus \big(HH_1^{\alpha}(A) \otimes A/[A,A]\big)$  since they are isomorphic by Proposition \ref{tensor abel}. In Particular, if $\alpha_A=\id_A$, the exact sequence in Theorem \ref{application} coincides with that of \cite[Theorem 7.4]{Gn}.

\item[b)]
From the exact sequence in Theorem \ref{application} and Theorem \ref{teor} immediately follows that the  vector spaces  $HH^{\alpha}_1(A)$ and $HH_1^M(A,\alpha_A)$ are isomorphic when the Hom-associative algebra $(A,\alpha_A)$ is perfect as Hom-Leibniz algebra and $HL_2^{\alpha}(A) = 0$.

\item[c)] $HH^{\alpha}_1(A)$ and $HH_1^M(A,\alpha_A)$ are also isomorphic for any commutative Hom-associative algebra  $(A,\alpha_A)$.
\end{enumerate}
\end{remark}


\section*{Acknowledgements}

First and second authors were supported by Ministerio de Econom\'{\i}a y Competitividad (Spain) (European FEDER support included), grant MTM2013-43687-P.  The second author was supported by Xunta de Galicia, grants EM2013/016 and GRC2013-045 (European FEDER support included), and  Shota Rustaveli National Science Foundation, grant FR/189/5-113/14.

\begin{center}

\end{center}

\end{document}